\documentclass[12pt]{article}
\usepackage[centertags]{amsmath}
\usepackage{amsfonts}
\usepackage{amssymb}
\usepackage{latexsym}
\usepackage{amsthm}
\usepackage{newlfont}
\usepackage{graphicx}
\usepackage{listings}
\usepackage{booktabs}
\usepackage{abstract}
\newlength{\defbaselineskip}
\newcommand{\setlinespacing}[1]%
           {\setlength{\baselineskip}{#1 \defbaselineskip}}

\newcommand{\N}{{\mathbb{N}}}

\newcommand{\actaqed}{\hfill $\actabox$}
{\medskip\noindent \textit{Proof of #1. }}%
{\actaqed \medskip}

\def \x{{\mathbf x}}
\def \w{{\mathbf w}}
\def \r{{\mathbf r}}
\def \u{{\mathbf u}}
\def\D{{\mathcal D}}

\def\R{{\mathbb R}}
\def \<{\langle}
\def\>{\rangle}

\def \e{\epsilon}
\def \wh{\widehat}
\def \de{\delta}
\def \dt{\delta}
\def \ff{\varphi}
\def\N{{\mathbb N}}
\def\Z{{\mathbb Z}}
\def \supp{\operatorname{supp}}
\def \sp{\operatorname{span}}

\def \exp{\operatorname{exp}}

\def\la{\lambda}

\def\Gm{\Gamma}

\newtheorem{Theorem}{Theorem}[section]
\newtheorem{Lemma}{Lemma}[section]
\newtheorem{Definition}{Definition}[section]
\newtheorem{Proposition}{Proposition}[section]
\newtheorem{Remark}{Remark}[section]

\newtheorem{Corollary}{Corollary}[section]
\numberwithin{equation}{section}

\begin{document}
\title{{Sparse approximation and recovery by greedy algorithms}\thanks{\it Math Subject Classifications.
primary:  41A65; secondary: 41A25, 41A46, 46B20.}}
\author{E. Livshitz\thanks{Evernote Corp. and Moscow State University. Research was supported in part by the Russian Foundation for Basic
Research (grants
11-01-00476 and 13-01-00554)} and V. Temlyakov \thanks{ University of South Carolina and Steklov Institute of Mathematics. Research was supported by NSF grant DMS-1160841 }} \maketitle
\begin{abstract}
{We study sparse approximation by greedy algorithms. Our contribution is two-fold. First, we prove exact recovery with high probability of random $K$-sparse signals within $\lceil K(1+\e)\rceil$ iterations of the Orthogonal Matching Pursuit (OMP). This result shows that in a probabilistic sense the OMP is almost optimal for exact recovery.  Second, we prove the Lebesgue-type inequalities for the Weak Chebyshev Greedy Algorithm,  a generalization of the Weak Orthogonal Matching Pursuit to the case of a Banach space. The main novelty of these results is a Banach space setting instead of a Hilbert space setting.
However, even in the case of a Hilbert space our results add some new elements to known results on the Lebesque-type inequalities for the RIP dictionaries. Our technique is a development of the recent technique created by Zhang.
 }
\end{abstract}


\section{Introduction}

This paper deals with sparse approximation. Driven by applications in biology, medicine, and engineering approximation problems are formulated in very high dimensions, which bring to the fore new phenomena.  One aspect of the high-dimensional context is a focus on  sparse signals (functions). The main motivation for the study of sparse signals is that many real world signals can be well approximated by sparse ones. A very important step in solving multivariate problems with large dimension
occurred  during last 20 years. Researchers began to use sparse
representations as a way to model the corresponding function classes. This approach
automatically implies a need for nonlinear approximation, in particular, for greedy approximation. We give a brief description of a sparse approximation problem. In a general setting we are working in a Banach space $X$ with a redundant system of elements $\D$ (dictionary $\D$). There is a solid justification of importance of a Banach space setting in numerical analysis in general and in sparse approximation in particular (see, for instance, \cite{Tbook}, Preface, and \cite{ST}).
An element (function, signal) $f\in X$ is said to be $K$-sparse with respect to $\D$ if
it has a representation $f=\sum_{i=1}^Kx_ig_i$,   $g_i\in \D$, $i=1,\dots,K$. The set of all $K$-sparse elements is denoted by $\Sigma_K(\D)$. For a given element $f_0$ we introduce the error of best $m$-term approximation
$$
\sigma_m(f_0,\D) := \inf_{f\in\Sigma_m(\D)} \|f_0-f\|.
$$
Here are two fundamental problems of sparse approximation.

{\bf P1. Exact recovery.} Suppose we know that $f_0\in \Sigma_K(\D)$. How can we recover it?

{\bf P2. Approximate recovery.}   How to design a practical algorithm that builds $m$-term approximations comparable to
best $m$-term approximations?

It is known that in both of the above problems greedy-type algorithms play a fundamental role. We discuss one of them
here. There are two special cases of the above general setting of the sparse approximation problem.

(I). Instead of a Banach space $X$ we consider a Hilbert space $H$. Approximation is still with respect to a redundant dictionary $\D$.

(II). We approximate in a Banach space $X$ with respect to a basis $\Psi$ instead of a redundant dictionary $\D$.

This section discusses setting (I) and the corresponding generalizations to the Banach space setting.   Section 4 addresses setting (II). We begin our discussion with the Orthogonal Greedy Algorithm (OGA) in a Hilbert space. The Orthogonal Greedy Algorithm is called the Orthogonal Matching Pursuit (OMP) in signal processing. We will use the name Orthogonal Matching Pursuit for this algorithm in this paper. It is natural to compare performance of the OMP with the best $m$-term approximation with regard to a dictionary $\D$.  We  recall some notations and
definitions from the theory of greedy algorithms. Let $H$ be a
real Hilbert space with an inner product
$\langle\cdot,\cdot\rangle $ and the norm $\|x\|:=\langle x,x
\rangle^{1/2}$. We say a set $\D$ of functions (elements) from $H$
is a dictionary if each $g\in \D$ has a unit norm $(\|g\|=1)$ and the closure of
$\sp \D$ is $H.$   Let a sequence $\tau = \{t_k\}_{k=1}^\infty$,
$0\le t_k \le 1$, be given. The following greedy algorithm was
defined in \cite{T1} under the name Weak Orthogonal Greedy Algorithm (WOGA).

 {\bf Weak Orthogonal Matching Pursuit (WOMP).} Let $f_0$ be given. Then for each $m\ge 1$ we inductively define:

(1) $\varphi_m  \in \D$ is any element satisfying
$$
|\langle f_{m-1},\varphi_m\rangle | \ge t_m
\sup_{g\in \D} |\langle f_{m-1},g\rangle |.
$$

(2) Let $H_m := \sp (\varphi_1,\dots,\varphi_m)$ and let
$P_{H_m}(\cdot)$ denote an operator of orthogonal projection onto $H_m$.
Define
$$
G_m(f_0,\D) := P_{H_m}(f_0).
$$

(3) Define the residual after $m$th iteration of the algorithm
$$
f_m := f_0-G_m(f_0,\D).
$$

In the case $t_k=1$, $k=1,2,\dots$,   WOMP is called the Orthogonal
Matching Pursuit (OMP). In this paper we only consider the case $t_k=t$, $k=1,2,\dots$, $t\in(0,1]$.

The theory of the WOMP is well developed (see \cite{Tbook}). In first results on performance of the WOMP in problems {\bf P1} and {\bf P2}
researchers imposed the incoherence assumption on a dictionary $\D$. The reader can find detailed discussion of these results in \cite{Tbook}, Section 2.6 and \cite{L}. Recently,  exact recovery results and Lebesgue-type inequalities for the WOMP under assumption that $\D$ satisfies Restricted Isometry Property (RIP) introduced in compressed sensing theory (see Definition \ref{RIP-def} below) have been proved (see \cite{Z}, \cite{F}, \cite{WS}).
A breakthrough result in this direction was obtained by Zhang \cite{Z}. In particular, he proved that if $\dt^{{RIP}}_{31K}(\D)<1/3$ then the OMP recovers exactly all $K$-sparse signals within  $30K$ iterations. In other words, $f_{30K}=0$. It is interesting and difficult problem to improve the constant $30$. There are several papers devoted to this problem (see \cite{F} and \cite{WS}).
In this paper we develop Zhang's technique in two directions: (1) to obtain exact recovery with high probability of random $K$-sparse signals
 within $\lceil K(1+\e)\rceil$ iterations of the OMP
 and (2) to obtain recovery results and the Lebesgue-type inequalities in the Banach space setting.

In Section 2 we prove exact recovery results under RIP conditions on a dictionary combined with assumptions on the sparse signal to be recovered (see Theorem \ref{T2.1}). We prove that the corresponding assumptions on a sparse signal are satisfied with high probability if it is a random signal. In particular, we prove the following theorem.
\begin{Theorem}\label{T1.1} For any $\e>0$ there exist $\de=\de(\e)>0$ and $K_0 = K_0(\e)$ such that for any dictionary $\D$, $\dt^{{RIP}}_{2K}(\D)<\dt$, $K\ge K_0$, the following statement holds. Let $f_0\in\Sigma_K(\D)$ and its nonzero coefficients are uniformly distributed on $[-1,1]$ independent random variables. Then $f_{\lceil K(1+\e)\rceil}=0$ with  probability greater than $1-\exp(-C(\e)K)$.
\end{Theorem}
This theorem shows that in a probabilistic sense the OMP is almost optimal for exact recovery.

Sections 3 is devoted to the Banach space setting.
Let $X$ be a Banach space with norm $\|\cdot\|:=\|\cdot\|_X$. As in the case of Hilbert spaces we say that a set of elements (functions) $\D$ from $X$ is a dictionary  if each $g\in \D$ has norm   one ($\|g\|=1$),
and the closure of $\sp \D$ is $X$.
For a nonzero element $g\in X$ we let $F_g$ denote a norming (peak) functional for $g$:
$$
\|F_g\|_{X^*} =1,\qquad F_g(g) =\|g\|_X.
$$
The existence of such a functional is guaranteed by the Hahn-Banach theorem.

Let
$\tau := \{t_k\}_{k=1}^\infty$ be a given weakness sequence of  nonnegative numbers $t_k \le 1$, $k=1,\dots$. We define the Weak Chebyshev Greedy Algorithm (WCGA) (see \cite{T15}) as a generalization for Banach spaces of the Weak Orthogonal Matching Pursuit.  We study in detail the WCGA in this paper.

 {\bf Weak Chebyshev Greedy Algorithm (WCGA).}
Let $f_0$ be given. Then for each $m\ge 1$ we have the following inductive definition.

(1) $\varphi_m :=\varphi^{c,\tau}_m \in \D$ is any element satisfying
$$
|F_{f_{m-1}}(\varphi_m)| \ge t_m\sup_{g\in\D}  | F_{f_{m-1}}(g)|.
$$

(2) Define
$$
\Phi_m := \Phi^\tau_m := \sp \{\varphi_j\}_{j=1}^m,
$$
and define $G_m := G_m^{c,\tau}$ to be the best approximant to $f_0$ from $\Phi_m$.

(3) Let
$$
f_m := f^{c,\tau}_m := f_0-G_m.
$$

In Section 3 we prove the Lebesgue-type inequalities for the WCGA. A very important advantage of the WCGA  is its convergence and rate of convergence properties. The WCGA is well defined for all $m$. Moreover, it is known (see \cite{T15} and \cite{Tbook}) that the WCGA with $\tau=\{t\}$ converges for all $f_0$ in all uniformly smooth Banach spaces with respect to any dictionary. That is, when $X$ is a real Banach space and the modulus of smoothness of $X$ is defined as follows
\begin{equation*}
\rho(u):=\frac{1}{2}\sup_{x,y;\|x\|= \|y\|=1}\left|\|x+uy\|+\|x-uy\|-2\right|,
\end{equation*}
then the uniformly smooth Banach space is the one with $\rho(u)/u\to 0$ when $u\to 0$.

For notational convenience we consider here a countable dictionary $\D=\{g_i\}_{i=1}^\infty$. For a given $f_0$, let the sparse element (signal)
 $$
 f:=f^\e=\sum_{i\in T}x_ig_i
 $$
 be such that $\|f_0-f^\e\|\le \e$ and $|T|=K$. For $A\subset T$ denote
 $$
 f_A:=f_A^\e := \sum_{i\in A}x_ig_i.
 $$

  We use the following two assumptions.

 {\bf A1. Nikol'skii-type inequality.} The sparse element $f=\sum_{i\in T}x_ig_i$ satisfies Nikol'skii-type $\ell_1X$ inequality with parameter $r$ if
 \begin{equation*}
 \sum_{i\in A} |x_i| \le C_1|A|^{r}\|f_A\|,\quad A\subset T,\quad r\ge 1/2.
 \end{equation*}

{\bf A2. Incoherence property.}  The sparse element $f=\sum_{i\in T}x_ig_i$ has incoherence property with parameters $D$ and $U$ if for any $A\subset T$ and any $\Lambda$, such that $A\cap \Lambda =\emptyset$ and $|A|+|\Lambda| \le D$, we have for any $\{c_i\}$
\begin{equation*}
\|f_A-\sum_{i\in\Lambda}c_ig_i\|\ge U^{-1}\|f_A\|.
\end{equation*}

 The main result of Section 3 is the following.
 \begin{Theorem}\label{T1.5} Let $X$ be a Banach space with $\rho(u)\le \gamma u^2$. Suppose $K$-sparse $f^\e$ satisfies {\bf A1}, {\bf A2} and $\|f_0-f^\e\|\le \e$. Then the WCGA with weakness parameter $t$ applied to $f_0$ provides
$$
\|f_{C(t,\gamma,C_1)U^2\ln (U+1) K^{2r}}\| \le C\e\quad\text{for}\quad K+C(t,\gamma,C_1)U^2\ln (U+1) K^{2r}\le D
$$
with an absolute constant $C$.
\end{Theorem}

  Theorem \ref{T1.5} provides a corollary for Hilbert spaces that gives sufficient conditions somewhat weaker than the known RIP conditions on $\D$ for the Lebesgue-type inequality to hold. We formulate it as a theorem.
 \begin{Theorem}\label{T1.6} Let $X$ be a Hilbert space. Suppose $K$-sparse $f^\e$ satisfies  {\bf A2} and $\|f_0-f^\e\|\le \e$. Then the WOMP with weakness parameter $t$ applied to $f_0$ provides
$$
\|f_{C(t,U) K}\| \le C\e\quad\text{for}\quad K+C(t,U) K\le D
$$
with an absolute constant $C$.
\end{Theorem}
Theorem \ref{T1.6} implies the following corollary.
 \begin{Corollary}\label{C1.0} Let $X$ be a Hilbert space. Suppose any $K$-sparse $f$ satisfies   {\bf A2}.   Then the WOMP with weakness parameter $t$ applied to $f_0$ provides
$$
\|f_{C(t,U) K}\| \le C\sigma_K(f_0,\D)\quad\text{for}\quad K+C(t,U) K\le D
$$
with an absolute constant $C$.
\end{Corollary}

We show in Sections 3 that the RIP condition with parameters $D$ and $\de$ implies the $(D,D)$ unconditionality with
$U=(1+\de)^{1/2}(1-\de)^{-1/2}$. Therefore, Corollary \ref{C1.0} reads as follows in this case.

 \begin{Corollary}\label{C1.1'} Let $X$ be a Hilbert space. Suppose $\D$ satisfies RIP condition with parameters $D$ and $\de$.   Then the WOMP with weakness parameter $t$ applied to $f_0$ provides
$$
\|f_{C(t,\de) K}\| \le C\sigma_K(f_0,\D)\quad\text{for}\quad K+C(t,\de) K\le D
$$
with an absolute constant $C$.
\end{Corollary}

  We   emphasize that in Theorem \ref{T1.5} we impose our conditions on an individual function $f^\e$. It may happen that the dictionary does not satisfy assumptions of $\ell_1X$ inequality and $(K,D)$-unconditionality (see Section 3) but the given $f_0$ can be approximated by $f^\e$ which does satisfy assumptions {\bf A1} and {\bf A2}. Even in the case of a Hilbert space our approach adds something new to the study based on the RIP. First of all, Theorem \ref{T1.6} shows that it is sufficient to impose assumption {\bf A2} on an individual $f^\e$ in order to obtain exact recovery and the Lebesgue-type inequality results. Second, Corollary \ref{C1.0} shows that the condition {\bf A2}, which is weaker than the RIP condition, is sufficient for exact recovery and the Lebesgue-type inequality results. Third, Corollary \ref{C1.1'} shows that even if we impose our assumptions in terms of RIP we do not need to assume that $\de < \de_0$. In fact, the result works for all $\de<1$ with parameters depending on $\de$.

  \section{Almost optimality of the OMP}

  We prove Theorem \ref{T1.1} in this section. For the readers convenience we use notations which are standard in signal processing.
Let $\D=\{\phi_i\}_{i=1}^N$ be a dictionary in $\R^M$, $M<N$. By $\Phi$ denote an $M\times N$ matrix, consisting of elements of $\D$ ($\phi_i\in \R^M$ is the $i$-th column of $\Phi$). We say that $\x\in \R^N$ is $S$-sparse if $\x$ has at most $S$ nonzero coordinates.

\begin{Definition}\label{RIP-def}
A matrix $\Phi$ satisfies $RIP(S,\dt)$ if the inequality
\begin{equation}\label{RIP-ineq}
(1-\dt)\|\x\|^2 \le \|\Phi \x\|^2\le (1+\dt)\|\x\|^2
 \end{equation}
 holds for all $S$-sparse $\x\in\R^N$.
 The minimum of all constants $\dt$, satisfying (\ref{RIP-ineq}),
 is called the isometric constant $\dt_S(\Phi) = \dt_S(\D) = \dt^{{RIP}}_S(\D)$.
\end{Definition}

In this section we study the OMP and use the ``compressed sensing notation'' for the residual of the OMP. Set
\begin{equation*}
 \r^m := f_m,\ m\ge 0.
\end{equation*}
Consider the set
\begin{equation*}
\Omega=\{1,\ldots,N\}.
\end{equation*}
Since $f_0\in\Sigma_K(\D)$, there exists an  $\x = (x_1,x_2,\ldots, x_N)$, $\supp\x = T$, $T\subset\Omega$, $|T| = K$ such that
\begin{equation*}
\r^0=f_0= \Phi\x.
\end{equation*}
Denote by $T^m$ the set of indices of $\phi_i$ picked by the OMP after $m$ iterations.
According to the definition of the OMP for every $m
\ge 0$ we choose $\x^m\in\R^N$, satisfying the following relations
\begin{equation}\label{Tmdef}
 \supp\x^m \subset T^m,\quad |T^m| = m,\text{ while }\r^m\ne 0,
\end{equation}
\begin{equation*}
G_m(f,\D) = \Phi\x^m,
\end{equation*}
\begin{equation}\label{rmPhi}
\r^m = \Phi\x-
\Phi\x^m.
\end{equation}
Let $N(\x,\nu)$ be the minimal integer such that
\begin{equation}\label{N-def}
\|\x_\Lambda\|^2 >\nu, \text { for all }\Lambda \subset T,\ |\Lambda| \ge N(\x,\nu) + 1.
\end{equation}

\begin{Theorem}\label{T2.1} There exists an absolute constant $\wh C$ such that for any $\dt$, $0<\dt < 0.001$, an integer $K \ge K_0 = K_0(\dt)$, and a dictionary $\D$, $\dt^{{RIP}}_{2K}(\D)<\dt$ the following statement holds.The OMP recovers exactly every $K$-sparse signal $\x$, $\|\x\|_\infty\le 1$, within  $K + 6N(\x, \wh C\dt^{1/2}K)$ iterations, in other words, $\r^{K + 6N(\x, \wh C\dt^{1/2}K)}=0$.
\end{Theorem}
Here is a direct corollary of Theorem \ref{T2.1}.
\begin{Corollary}\label{C2.1} Let $K$-sparse $\x$ be such that $|x_i|=1$, $i\in T$, $|T|=K$. Then under assumptions of Theorem \ref{T2.1} the OMP recovers $\x$ exactly within $(1+ 6\wh C\dt^{1/2})K$ iterations.
\end{Corollary}
\begin{proof}
Set
\begin{equation*}
\Gm^m:=T\setminus T^m.
\end{equation*}
 We fix
 \begin{equation}\label{a-def}
a := \dt^{1/2}.
\end{equation}
Consider $m\in\Z_+ $ such that
\begin{equation}\label{bas-ineq-l}
 |T^m| +[aK] = m+[aK] \le K.
 \end{equation}
Assume that $K\ge K_0 = K_0(a) \ge 1/a$.
Let $z^m$ be the maximal number, satisfying the following inequality
\begin{equation}\label{zl-def}
\left|\{i\in \Gm^m:|x_i| \ge z^m\}\right| \ge [aK],\ \left|\{i\in \Gm^m: |x_i| \le z^m\}\right| \ge |\Gm^m|- [aK].
\end{equation}
In other words $z^m$ is the $[aK]$th largest element out of $\{|x_i|\}_{i\in \Gamma^m}$. We use the following lemma.

\begin{Lemma}\label{lm-main}
Under (\ref{bas-ineq-l})  the following inequality is valid:
\begin{equation*}
\|\r^m\|^2 - \|\r^{m+1}\|^2 \ge (z^m)^2(1-C_1 a).
 \end{equation*}
\end{Lemma}

\begin{proof}
According to (\ref{zl-def}), we can choose sets
\begin{equation}\label{Tl+cardinality}
\Gm^m_+ \subset \Gm^m,\quad |\Gm^m_+| = [aK]
\end{equation}
and
\begin{equation}\label{Tl-def}
\Gm^m_-:=\Gm\setminus \Gm^m_+
\end{equation}
with the following property
\begin{equation}\label{Tl+def}
\min_{i\in T^m_+}|x_i|\ge  z^m \ge \max_{i\in \Gm^m_-}|x_i|.
\end{equation}
Consider $\w\in\R^N$ such that
\begin{equation}\label{w-def}
\w_{T\cap T^m\cup \Gm^m_+} = \x_{T\cap T^m\cup \Gm^m_+},\quad \w_{\Omega\setminus(T\cap T^m\cup \Gm^m_+)} = 0.
\end{equation}
We use several well-known properties of the OMP:
\begin{equation}\label{WS-1}
\|\r^m\|^2 - \|\r^{m+1}\|^2 \ge \sup_{\phi\in\D}\<\r^m,\phi\>^2,
 \end{equation}
\begin{equation}\label{Prp-1}
\sup_{\phi\in\D}|\<\r^m,\phi\>|\ge \frac{|\<r^m,\Phi\u\>|}{\|\u\|_1},\quad \u\in\R^\N,
 \end{equation}
\begin{equation}\label{Prp-2}
\<\r^m, \Phi{{\u}}\> = 0, \mbox{ if }\supp{{\u}}\subset T^m.
 \end{equation}
 In particular
 \begin{equation}\label{Prp-3}
\<\r^m, \Phi\x^m\> = 0.
 \end{equation}
 Using (\ref{Prp-1}) for $\u = \w_{\Omega\setminus T^m}$, we can estimate
\begin{eqnarray*}
|\sup_{\phi\in\D}\<\r^m,\phi\>| & \ge & \frac{|\<\r^m,\Phi\w_{\Omega\setminus T^m}\>|}{\|\w_{\Omega\setminus T^m}\|_1} \stackrel{(\ref{Prp-2})}{=} \frac{|\<\r^m,\Phi\w\>|}{\|\w_{\Omega\setminus T^m}\|_1}\\
& \stackrel{(\ref{Prp-3})}{=} &
\frac{|\<\r^m,\Phi(\w-\x^m)\>|}{\|\w_{\Omega\setminus T^m}\|_1} \ge \frac{|\<\r^m,\Phi(\w-x^m)\>|}{\|\w_{\Omega\setminus T^m}\|_0^{1/2}\|\w_{\Omega\setminus T^m}\|_2}.
 \end{eqnarray*}
Applying (\ref{Tl+cardinality}) and (\ref{w-def}), we obtain from the above inequality
\begin{equation}\label{WS-2mdf}
|\sup_{\phi\in\D}\<\r^m,\phi\>| \ge
\frac{|\<\r^m,\Phi(\w-x^m)\>|}{(|\Gm^m_+|)^{1/2}\|\w_{\Omega\setminus T^m}\|_2} \ge \frac{|\<\r^m,\Phi(\w-x^m)\>|}{(aK)^{1/2}\|\w_{\Omega\setminus T^m}\|_2}.
 \end{equation}
We estimate
\begin{eqnarray*}
\|\r^m\|^2  & = & \|\Phi(\x-\x^m)\|^2 \stackrel{RIP}{\ge} (1-\dt)\|\x-\x^m\|^2\ge(1-\dt)\|(\x-\x^m)_{\Gm^m}\|^2 \\
&= & (1-\dt)\|\x_{\Gm^m}\|^2
= (1-\dt)\|\x_{\Gm^m_+\cup \Gm^m_-}\|^2
\\
& = &
(1-\dt)(\|\x_{\Gm^m_+}\|^2 + \|\x_{\Gm^m_-}\|^2)\stackrel{(\ref{Tl+def}),(\ref{Tl+cardinality})}{\ge} (1-\dt)((z^m)^2[aK] + \|\x_{\Gm^m_-}\|^2),
 \end{eqnarray*}
and
\begin{equation*}
\|\Phi(\w-\x)\|^2\stackrel{(\ref{w-def})}{=}\|\Phi\x_{\Gm^m_-}\|^2\stackrel{RIP}{\le} (1+\dt)\|\x_{\Gm^m_-}\|^2.
 \end{equation*}
Combining two last inequalities, we obtain, for sufficiently large $K_0=K_0(a)=K_0(\e)$,
\begin{eqnarray}
\|\r^m\|^2 - \|\Phi(\w-\x)\|^2 & \ge & (1-\dt)(z^m)^2[aK] -2\dt \|\x_{\Gm^m_-}\|^2 \nonumber \\
& \stackrel{(\ref{Tl+def})}{\ge} & (1-\dt)(z^m)^2[aK] - 2\dt (z^m)^2|\Gm^m_-| \nonumber \\
& \stackrel{(\ref{Tl-def})}{\ge} &
 (1-\dt)(z^m)^2[aK] - 2\dt (z^m)^2K \nonumber\\
 &\ge & (1-2\dt)(z^m)^2(aK) - 2\dt (z^m)^2K \nonumber\\
& = &(z^m)^2aK((1-2\dt)-\frac{2\dt}{a}) \nonumber\\
& \stackrel{(\ref{a-def})}{\ge} &  (z^m)^2aK(1- 4a). \label{diff-est}
 \end{eqnarray}

Following the technique from~\cite{WS}
we have
\begin{eqnarray}
|\<\r^m,\Phi(\w-\x^m)\>| & = &
\frac{1}{2}\left|\|\Phi(\w-\x^m)\|_2^2+ \|\r^m\|^2 -  \|\Phi(\w-\x^m) - r^m\|^2\right|\nonumber\\
& \stackrel{(\ref{rmPhi})}{=}&
\frac{1}{2}\left|\|\Phi(\w-\x^m)\|_2^2+ (\|\r^m\|^2 - \|\Phi(\w-\x)\|^2)\right|\nonumber\\
& \stackrel{(\ref{diff-est})}{\ge}&
\left(\|\Phi(\w-\x^m)\|_2^2(z^m)^2aK(1-4a)\right)^{1/2}\nonumber \\
& = & \|\Phi(\w-\x^m)\|_2 z^m(aK(1-4a))^{1/2}\nonumber\\
& \stackrel{RIP}{\ge} & (1-\dt)^{1/2}\|\w-\x^m\|z^m(aK(1-4a))^{1/2}\nonumber\\
& \ge & \|\w-\x^m\|z^m(aK)^{1/2}(1-c_1 a)\nonumber\\
& \ge &
\|(\w-\x^m)_{\Omega\setminus T^m}\|z^m(aK)^{1/2}(1-c_1 a) \stackrel{(\ref{Tmdef})}{=}\nonumber\\
& = &\|\w_{\Omega\setminus T^m}\|z^m(aK)^{1/2}(1-c_1 a). \label{WS-3}
  \end{eqnarray}

Substituting  (\ref{WS-3}) in (\ref{WS-2mdf}) and (\ref{WS-1}), we finally get
\begin{eqnarray*}
\|\r^m\|^2 - \|\r^{m+1}\|^2 & \ge & \frac{\<\r^m,\Phi(\w-\x^m)\>^2}{aK\|\w_{\Omega\setminus T^m}\|_2^2}\ge\frac{\|\w_{\Omega\setminus T^m}\|^2(z^m)^2aK(1-c_1a)^2}{aK\|\w_{\Omega\setminus T^m}\|_2^2}\\
& \ge &
(z^m)^2(1- C_1a).
 \end{eqnarray*}
 \end{proof}

 We continue to prove Theorem \ref{T2.1}.
Without loss of generality we may assume that $T=\{1,\ldots, K\}$ and that the sequence $\{|x_i|\}_{i = 1}^K$ decreases.
Then using the inequality $|T\cap T^m|\le m$ and the definition (\ref{zl-def}), we have
 \begin{equation*}
 z^m\ge |x_{m + [aK]}|\ge |x_{m + 1 + [aK]}|.
 \end{equation*}
Applying Lemma~\ref{lm-main}, we have for $m\ge 1$, $m + [aK] \le K$,
\begin{equation}\label{fromLmMain}
 \|\r^{m-1}\|^2 - \|\r^{m}\|^2 \ge (z^{m-1})^2(1-C_1a)\ge x_{m+[aK]}^2(1-C_1a).
 \end{equation}

First we bound $\|\r^K\|$ from above
 \begin{eqnarray}\label{rKUpper}
 \|\r^K\|^2 & = & \|\r^0\|^2 - \sum_{m = 1}^K\left(\|\r^{m-1}\|^2 - \|\r^{m}\|^2\right)\nonumber\\
   & \stackrel{(\r^0=\Phi\x)}{=} & \|\Phi \x\|^2 - \sum_{m=1}^K\left(\|\r^{m-1}\|^2 - \|\r^{m}\|^2\right)\nonumber \\
   & \stackrel{RIP}{\le} &
(1+\dt)\sum_{i=1}^K x_i^2 - \sum_{m=1}^{K-[aK]}\left(\|\r^{m-1}\|^2 - \|\r^{m}\|^2\right)\nonumber\\
 & \stackrel{(\ref{fromLmMain})}{\le} & (1+\dt)\sum_{i=1}^K x_i^2 - \sum_{m=1}^{K-[aK]}x_{m+[aK]}^2(1-C_1a)\nonumber\\
& \le & (1+\dt)\sum_{i=1}^K x_i^2 - \sum_{i=1 + [aK]}^{K}x_{i}^2(1-C_1a)\nonumber\\
& \le &
(\dt + C_1a) \sum_{i=1}^K x_i^2  +\sum_{i=1}^{[aK]} x_i^2 \nonumber \\
&\stackrel{|x_i|\le 1}{\le}& (\dt + C_1a)K + [aK] \le KC_2 a \stackrel{(\ref{a-def})}{=}KC_2\dt^{1/2} .
 \end{eqnarray}
Then using RIP, we can estimate $\|\r^K\|$ from below
 \begin{equation}\label{rKLower}
 \|\r^K\|^2 \ge (1-\dt)\sum_{i\in T\setminus T^K}x_i^2.
  \end{equation}
Set
\begin{equation*}
\wh C:=\frac{C_2}{1-\dt}.
  \end{equation*}
Combining this definition with (\ref{rKUpper}) and (\ref{rKLower}), we obtain
 \begin{equation*}
\sum_{i\in T\setminus T^K}x_i^2 \le \wh C K\dt^{1/2}.
  \end{equation*}
Thus, using (\ref{N-def}), we conclude that
 \begin{equation}\label{TTK-ineq}
|T\setminus T^K|\le N(\x,\wh C K\dt^{1/2}).
  \end{equation}

It is known that (see Lemma 1.2 from \cite{CT} and Lemma 1 from \cite{CWX})
\begin{equation*}
 \dt_{2S}(\Phi)\le 3 \dt_{S}(\Phi).
\end{equation*}
Then the condition $\dt < 0.001$ implies that
 \begin{equation}\label{dt10K}
 \dt_{10K}(\Phi)\le \dt_{16K}(\Phi) \le 27 \dt_{2K}(\Phi)   \le 27\dt\le 0.03.
\end{equation}
Now we can apply the improvement of Zhang's theorem
obtained by Wang and Shim (\cite{WS}, Theorem 3.1). It claims that
under (\ref{dt10K}) we have
 \begin{equation*}
\r^{K + 6 |T\setminus T^K|} = 0,
\end{equation*}
Therefore, taking into account (\ref{TTK-ineq}), we finally get
 \begin{equation*}
\r^{K + 6 N(\x,\wh C K\dt^{1/2})} = 0.
\end{equation*}

\end{proof}

As corollaries of Theorem~\ref{T2.1} we obtain Theorem~\ref{T1.1} and the following result.

\begin{Theorem}\label{T2.2} For any $\e_1,\e_2>0$ there exist $\de=\de(\e_1,\e_2)>0$ and $K_0 = K_0(\e_1,\e_2)$ such that for any dictionary $\D$, $\dt^{{RIP}}_{2K}(\D)<\dt$, $K\ge K_0$, the following statement holds. If $\r^0=f_0 \in\Sigma_K(\D)$ and its nonzero coefficients belong to $[-1,1]\setminus(-\e_1,\e_1)$, then $\r^{\lceil K(1+\e_2)\rceil}=0$.
\end{Theorem}

\begin{proof}
It is clear that for any $\Lambda\subset T$ we have
 \begin{equation*}
\|\x_\Lambda\|^2 \ge \e_1^2|\Lambda|.
\end{equation*}
Hence according to~(\ref{N-def}) we get
 \begin{equation*}
N(\x,\nu) \le\frac{\nu}{\e_1^2}.
\end{equation*}
Then
 \begin{equation*}
6N(\x,\wh C K \dt^{1/2}) \le K6\frac{\wh C \dt^{1/2}}{\e_1^2}.
\end{equation*}
Thus, to complete the proof it remains to choose $\dt=\dt(\e_1,\e_2)$ such that
 \begin{equation*}
6\frac{\wh C \dt^{1/2}}{\e_1^2} \le \e_2.
\end{equation*}
\end{proof}

\begin{Lemma}\label{lm-prob} Assume that $p<1$ and numbers $x_i$, $1\le i\le K$, $K\ge K_0(p)$  are uniformly distributed on $[-1,1]$ independent random variables. Then
\begin{equation*}
\left|\{i: |x_i| <p\}\right| \le 2pK
\end{equation*}
 with  probability greater than $1-\exp(-C(p)K)$.
\end{Lemma}
\begin{proof}
For $i$, $1\le i\le K$, we set $\xi_i=0$, if $|x_i|\ge p$, and $\xi_i = 1$, otherwise.
So $\xi_i$ has Bernoulli distribution with
\begin{equation*}
\texttt{P}\{\xi_i=1\} = p,\quad \texttt{P}\{\xi_i=0\}=1-p,\quad \texttt{E}\xi_i = p.
\end{equation*}
By Hoeffding's inequality (see, for instance, \cite{Tbook}, p. 197) we obtain
$$
\texttt{P}\left\{\left|\frac{1}{K}\sum_{i=1}^K\xi_i-p\right|\ge p\right\} \le 2\exp(-Kp^2/2).
$$
Clearly,
$$
|\{i:  |x_i| <p\}| = \sum_{i=1}^K \xi_i.
$$
Therefore,
\begin{equation*}
\texttt{P}\left\{\left|\{i: |x_i| <p\}\right| \le 2p\right\}=\texttt{P}\left\{\frac{1}{K}\sum_{i=1}^K\xi_i\le 2p\right\} \ge 1- 2\exp(-Kp^2/2).
\end{equation*}

\end{proof}

We now give a proof of Theorem~\ref{T1.1} from the Introduction.
\begin{proof}
Let
\begin{equation*}
\varkappa := \varkappa(\dt):= \wh C \dt^{1/2}.
\end{equation*}
According to Lemma~\ref{lm-main} with probability greater than $1-\exp(-C(\dt)K)$ we have
\begin{equation}\label{derivedfromprob}
\left|\{i: |x_i| <\varkappa^{1/3}\}\right| \le 2\varkappa^{1/3}K
\end{equation}
To prove the theorem we need to estimate $N(\x,\varkappa K)$. Consider $\Lambda\subset T$ such that
\begin{equation*}
\|\x_\Lambda\|^2\le\varkappa K.
\end{equation*}
Then we estimate
\begin{eqnarray*}
|\Lambda|& = & |\{i\in\Lambda: |x_i| < \varkappa^{1/3}\}| + |\{i\in\Lambda: |x_i| \ge \varkappa^{1/3}\}|\\
& \le &
|\{i\in\Lambda: |x_i| < \varkappa^{1/3}\}| +\frac{\varkappa K}{(\varkappa^{1/3})^2} \stackrel{(\ref{derivedfromprob})}{\le}
2\varkappa^{1/3}K + \varkappa^{1/3}K=3\varkappa^{1/3}K.
\end{eqnarray*}
Therefore, by definition (\ref{N-def}) we have
\begin{equation*}
  N(\x,\varkappa K)\le 3\varkappa^{1/3}K= 3(\wh C)^{1/3}\dt^{1/6}K.
\end{equation*}
To complete the proof it remains to apply Theorem~\ref{T2.1} for $\dt<0.001$ providing
\begin{equation*}
N(\x,\varkappa K) <\e K/6.
\end{equation*}

\end{proof}

\section{Lebesgue-type inequalities}

 We discuss here the Lebesgue-type inequalities for the WCGA with $\tau =\{t\}$, $t\in(0,1]$.   We repeat the above assumptions {\bf A1} and {\bf A2} with remarks on the corresponding properties of dictionaries.
 For a given $f_0$ let sparse element (signal)
 $$
 f:=f^\e=\sum_{i\in T}x_ig_i
 $$
 be such that $\|f_0-f^\e\|\le \e$ and $|T|=K$. For $A\subset T$ denote
 $$
 f_A:=f_A^\e := \sum_{i\in A}x_ig_i.
 $$

 Here are two assumptions that we will use.

 {\bf A1.} We say that $f=\sum_{i\in T}x_ig_i$ satisfies the Nikol'skii-type $\ell_1X$ inequality with parameter $r$ if
 \begin{equation}\label{C1}
 \sum_{i\in A} |x_i| \le C_1|A|^{r}\|f_A\|,\quad A\subset T,\quad r\ge 1/2.
 \end{equation}
 We say that a dictionary $\D$ has the Nikol'skii-type $\ell_1X$ property with parameters $K$, $r$   if any $K$-sparse element satisfies the Nikol'skii-type
 $\ell_1X$ inequality with parameter $r$.

{\bf A2.}  We say that $f=\sum_{i\in T}x_ig_i$ has incoherence property with parameters $D$ and $U$ if for any $A\subset T$ and any $\Lambda$ such that $A\cap \Lambda =\emptyset$, $|A|+|\Lambda| \le D$ we have for any $\{c_i\}$
\begin{equation}\label{C2}
\|f_A-\sum_{i\in\Lambda}c_ig_i\|\ge U^{-1}\|f_A\|.
\end{equation}
We say that a dictionary $\D$ is $(K,D)$-unconditional with a constant $U$ if for any $f=\sum_{i\in T}x_ig_i$ with
$|T|\le K$ inequality (\ref{C2}) holds.

The term {\it unconditional} in {\bf A2} is justified by the following remark. The above definition of $(K,D)$-unconditional dictionary is equivalent to the following definition. Let $\D$ be such that any subsystem of $D$ distinct elements $e_1,\dots,e_D$ from $\D$ is linearly independent and for any $A$ with $|A|\le K$ and any coefficients $\{c_i\}$ we have
$$
\|\sum_{i\in A}c_ie_i\| \le U\|\sum_{i=1}^Dc_ie_i\|.
$$

Let $\D$ be the Riesz dictionary with depth $D$ and parameter $\delta\in (0,1)$. This class of dictionaries is a generalization of the class of classical Riesz bases. We give a definition in a general Hilbert space (see \cite{Tbook}, p. 306).
\begin{Definition}\label{D3.1} A dictionary $\D$ is called the Riesz dictionary with depth $D$ and parameter $\delta \in (0,1)$ if, for any $D$ distinct elements $e_1,\dots,e_D$ of the dictionary and any coefficients $a=(a_1,\dots,a_D)$, we have
\begin{equation*}
(1-\de)\|a\|_2^2 \le \|\sum_{i=1}^D a_ie_i\|^2\le(1+\de)\|a\|_2^2.
\end{equation*}
We denote the class of Riesz dictionaries with depth $D$ and parameter $\delta \in (0,1)$ by $R(D,\de)$.
\end{Definition}
It is clear that the term Riesz dictionary with depth $D$ and parameter $\delta \in (0,1)$ is another name for a dictionary satisfying the Restricted Isometry Property with parameters $D$ and $\de$. The following simple lemma holds.
\begin{Lemma}\label{L3.1} Let $\D\in R(D,\de)$ and let $e_j\in \D$, $j=1,\dots, s$. For $f=\sum_{i=1}^s a_ie_i$ and $A \subset \{1,\dots,s\}$ denote
$$
S_A(f) := \sum_{i\in A} a_ie_i.
$$
If $s\le D$ then
$$
\|S_A(f)\|^2 \le (1+\de)(1-\de)^{-1} \|f\|^2.
$$
\end{Lemma}

Lemma \ref{L3.1} implies that if $\D\in R(D,\de)$ then it is $(D,D)$-unconditional with a constant $U=(1+\de)^{1/2}(1-\de)^{-1/2}$.

We need the concept of cotype of a Banach space $X$. We say that $X$ has cotype $q\ge 2$ if for any finite number of elements $u_i\in X$ we have the inequality
$$
\left(\text{Average}_{\pm}\|\sum_{i}\pm u_i \|^q\right)^{1/q} \ge C_q\left(\sum_{i}\|u_i\|^q\right)^{1/q}.
$$
It is known that the $L_p$ spaces with $2\le p<\infty$ have cotype $q=p$ and $L_p$ spaces with $1< p\le 2$ have cotype $2$.
\begin{Remark}\label{R3.1} Suppose $\D$ is $(K,K)$-unconditional with a constant $U$. Assume that $X$ is of cotype $q$ with a constant $C_q$. Then $\D$ has the Nikol'skii-type $\ell_1X$ property with parameters $K,1-1/q$ and $C_1=2UC_q^{-1}$.
\end{Remark}
\begin{proof} Our assumption about $(K,K)$-unconditionality implies: for any $A$, $|A|\le K$, we have
$$
\|\sum_{i\in A} \pm x_i g_i\| =\|\sum_{i\in A_+}x_i g_i - \sum_{i\in A_-}x_i g_i\| \le 2U\|\sum_{i\in A} x_i g_i\|.
$$
Therefore, by $q$-cotype assumption
$$
\|\sum_{i\in A} x_i g_i\|^q \ge (2U)^{-q}C_q^q\sum_{i\in A}|x_i|^q.
$$
This implies
$$
\sum_{i\in A} |x_i| \le |A|^{1-1/q}\left(\sum_{i\in A}|x_i|^q\right)^{1/q} \le 2UC_q^{-1}|A|^{1-1/q}\|\sum_{i\in A} x_i g_i\|.
$$
\end{proof}
The above proof also gives the following individual function version of Remark \ref{R3.1}.
\begin{Remark}\label{R3.2} Suppose $f=\sum_{i\in T}x_ig_i$ has incoherence property with parameters $D$ and $U$. Assume that $X$ has cotype $q$ with a constant $C_q$. Then $f$ satisfies the Nikol'skii-type $\ell_1X$ inequality with parameter $r=1-1/q$ and $C_1=2UC_q^{-1}$.
\end{Remark}
It is known that a Hilbert space has cotype $2$. Therefore, Remark \ref{R3.2} shows that assumption {\bf A2} implies assumption {\bf A1} with $r=1/2$. This explains how Theorem \ref{T1.6} is derived from Theorem \ref{T1.5}.

We note that the $(K,CK)$-unconditionality assumption on the dictionary $\D$ in a Hilbert space $H$ is somewhat weaker than the assumption $\D\in R(CK,\de)$. Also, our theorems do not assume that the dictionary satisfies
assumptions {\bf A1} and {\bf A2}; we only assume that the individual function $f$,   a $K$-sparse approximation of a given $f_0$, satisfies {\bf A1} and {\bf A2}.

In assumption (\ref{C2}) we always have $U\ge 1$. In the extreme case $U=1$ assumption (\ref{C2}) is a strong assumption that leads to strong results.
\begin{Proposition}\label{P3.1} Let $X$ be a uniformly smooth Banach space. Assume that $f=\sum_{i\in T}x_ig_i$,  $|T|=K$, and the set of indices $T$ has the following property. For any $g\in\D$ distinct from $g_i$, $i\in T$, and any $c_i$, $c$ we have
\begin{equation}\label{3.2}
\|\sum_{i\in T}c_ig_i-cg\|\ge \|\sum_{i\in T}c_ig_i\|.
\end{equation}
Then the WCGA with $t_k\neq 0$, $k=1,2,\dots$, recovers $f$ exactly after $K$ iterations.
\end{Proposition}
\begin{proof} It is known (see, for instance, \cite{Tbook}, Lemma 6.9, p. 342) that (\ref{3.2}) implies
$$
F_f(g)=0,\qquad g\in \D\setminus\{g_i\}_{i\in T}.
$$
Thus, at the first iteration the WCGA picks $\varphi_1\in \{g_i\}_{i\in T}$. Then $f_1$ has the form $\sum_{i\in T}c_ig_i$ and we repeat the above argument. Then $\varphi_2\in \{g_i\}_{i\in T}\setminus \{\varphi_1\}$. After $K$ iterations all $g_i$, $i\in T$, will be taken and therefore we will have $f_K=0$.
\end{proof}

Proposition \ref{P3.1} can be applied in the following situation. Assume that $\Psi=\{\psi_i\}_{i=1}^\infty$ is a monotone basis for a uniformly smooth Banach space $X$. Then any $f=\sum_{i=1}^K x_i\psi_i$ will be recovered by the WCGA after $K$ iterations.
In particular, this applies to the Haar basis in $L_p$, $1<p<\infty$.

We now proceed to main results of this section.

\begin{Theorem}\label{T3.1} Let $X$ be a Banach space with $\rho(u)\le \gamma u^2$. Suppose for a given $f_0$ we have $\|f_0-f^\e\|\le \e$ with $K$-sparse $f:=f^\e$ satisfying {\bf A1} and {\bf A2}. Then for any $k\ge 0$ we have for $K+m \le D$
$$
\|f_m\| \le \|f_k\|\exp\left(-\frac{c_1(m-k)}{K^{2r}}\right) +2\e,
$$
where $c_1:= \frac{t^2}{32\gamma C_1^2 U^2}$.
\end{Theorem}
\begin{proof}  Let
$$
 f:=f^\e=\sum_{i\in T}x_ig_i,\quad |T|=K,\quad g_i\in \D.
 $$
 Denote by $T^m$ the set of indices of $g_i$ picked by the WCGA after $m$ iterations, $\Gamma^m := T\setminus T^m$.
 Denote by $A_1(\D)$ the closure in $X$ of the convex hull of the symmetrized dictionary $\D^\pm:=\{\pm g,g\in D\}$.   We will bound  $\|f_m\|$ from above.  Assume $\|f_{m-1}\|\ge \e$. Let $m>k$. We bound from below
 $$
 S_m:=\sup_{\phi\in A_1(\D)} |F_{f_{m-1}}(\phi)|.
 $$
 Denote $A_m:=\Gamma^{m-1}$. Then
 $$
 S_m \ge F_{f_{m-1}}(f_{A_m}/\|f_{A_m}\|_1),
 $$
 where $\|f_A\|_1 := \sum_{i\in A} |x_i|$. Next, by Lemma 6.9, p. 342, from \cite{Tbook} we obtain
 $$
 F_{f_{m-1}}(f_{A_m}) = F_{f_{m-1}}(f^\e) \ge \|f_{m-1}\|-\e.
 $$
 Thus
 $$
 S_m\ge \|f_{A_m}\|^{-1}_1(\|f_{m-1}\|-\e).
 $$
 By (\ref{C1}) we get
 $$
 \|f_{A_m}\|_1 \le C_1 |A_m|^{r}\|f_{A_m}\| \le C_1 K^{r}\|f_{A_m}\| .
 $$
 Then
 \begin{equation}\label{3.4}
 S_m \ge \frac{\|f_{m-1}\|-\e}{C_1 K^{r} \|f_{A_m}\|}.
 \end{equation}
 From the definition of the modulus of smoothness we have for any $\la$
\begin{equation}\label{3.4'}
\|f_{m-1}-\la \varphi_m \| +\|f_{m-1}+\la \varphi_m \| \le 2\|f_{m-1}\|\left(1+\rho\left(\frac{\la}{\|f_{m-1}\|}\right)\right)
\end{equation}
and by (1) from the definition of the WCGA and Lemma 6.10 from \cite{Tbook}, p. 343, we get
$$
|F_{f_{m-1}}(\varphi_m)| \ge t \sup_{g\in \D} |F_{f_{m-1}}(g)| =
$$
$$
t\sup_{\phi\in A_1(\D)} |F_{f_{m-1}}(\phi)|=tS_m.
$$
Then either $F_{f_{m-1}}(\varphi_m)\ge tS_m$ or $F_{f_{m-1}}(-\varphi_m)\ge tS_m$. Both cases are treated in the same way. We demonstrate the case $F_{f_{m-1}}(\varphi_m)\ge tS_m$. We have for $\la \ge 0$
$$
\|f_{m-1}+\la \varphi_m\| \ge F_{f_{m-1}}(f_{m-1}+\la\varphi_m)\ge  \|f_{m-1}\| + \la tS_m.
$$
From here and from (\ref{3.4'}) we obtain
 $$
 \|f_m\| \le \|f_{m-1}-\la \varphi_m\| \le \|f_{m-1}\| +\inf_{\la\ge 0} (-\la t S_m + 2\|f_{m-1}\|\rho(\la/\|f_{m-1}\|)).
 $$
 We discuss here the case $\rho(u) \le \gamma u^2$. Using (\ref{3.4}) we get
 $$
 \|f_m\| \le \|f_{m-1}\|\left(1-\frac{\la t}{C_1K^{r}\|f_{A_m}\|} +2\gamma \frac{\la^2}{\|f_{m-1}\|^2}\right) + \frac{\e \la t}{C_1K^{r}\|f_{A_m}\|}.
 $$
 Let $\la_1$ be a solution of
 $$
 \frac{\la t}{2C_1K^{r}\|f_{A_m}\|} = 2\gamma \frac{\la^2}{\|f_{m-1}\|^2},\quad \la_1 = \frac{t\|f_{m-1}\|^2}{4\gamma C_1 K^{r}\|f_{A_m}\|}.
 $$
 Our assumption (\ref{C2}) gives
 \begin{eqnarray*}
 \|f_{A_m}\|&=&\|(f^\e-G_{m-1})_{A_m}\| \le U\|f^\e-G_{m-1}\|\\
 &\le&U(\|f_0-G_{m-1}\|+\|f_0-f^\e\|) \le U(\|f_{m-1}\|+\e).
 \end{eqnarray*}
 Specify
 $$
 \la = \frac{t \|f_{A_m}\|}{16\gamma C_1 K^{r} U^2}.
 $$
 Then, using $\|f_{m-1}\| \ge \e$ we get
 $$
 \frac{\la}{\la_1} = \frac{\|f_{A_m}\|^2}{4\|f_{m-1}\|^2} \le 1
 $$
  and  obtain
 $$
 \|f_m\| \le \|f_{m-1}\| \left( 1- \frac{t^2}{32\gamma C_1^2 U^2 K^{2r}}\right) + \frac{\e t^2}{16\gamma C_1^2 U^2K^{2r}}.
 $$
 Denote $c_1:= \frac{t^2}{32\gamma C_1^2 U^2}$. Then
 $$
 \|f_m\| \le \|f_{k}\|\exp\left(-\frac{c_1(m-k)}{K^{2r}}\right) + 2\e.
 $$
\end{proof}

\begin{Theorem}\label{T3.2} Let $X$ be a Banach space with $\rho(u)\le \gamma u^2$. Suppose $K$-sparse $f^\e$ satisfies {\bf A1}, {\bf A2} and $\|f_0-f^\e\|\le \e$. Then the WCGA with weakness parameter $t$ applied to $f_0$ provides
$$
\|f_{C(t,\gamma,C_1)U^2\ln (U+1) K^{2r}}\| \le CU\e\quad\text{for}\quad K+C(t,\gamma,C_1)U^2\ln (U+1) K^{2r}\le D
$$
with an absolute constant $C$ and $C(t,\gamma,C_1) = C_2\gamma C_1^2 t^{-2}$.
\end{Theorem}
We formulate an immediate corollary of Theorem \ref{T3.2} with $\e=0$.
\begin{Corollary}\label{C3.1} Let $X$ be a Banach space with $\rho(u)\le \gamma u^2$. Suppose $K$-sparse $f$ satisfies {\bf A1}, {\bf A2}. Then the WCGA with weakness parameter $t$ applied to $f$ recovers it exactly after $C(t,\gamma,C_1)U^2\ln (U+1) K^{2r}$ iterations under condition $K+C(t,\gamma,C_1)U^2\ln (U+1) K^{2r}\le D$.
\end{Corollary}

 \begin{proof} We use the above notations
   $T^m $ and  $\Gamma^m := T\setminus T^m$. Let $k\ge 0$ be fixed. Suppose
 $$
 2^{n-1} <|\Gamma^k| \le 2^n.
 $$
 For $j=1,2,\dots,n,n+1$ consider the following pairs of sets $A_j,B_j$: $A_{n+1}=\Gamma^k$, $B_{n+1}=\emptyset$; for $j\le n$, $A_j:=\Gamma^k\setminus B_j$ with $B_j\subset \Gamma^k$ is such that $|B_j|\ge |\Gamma^k|-2^{j-1}$ and for any set $J\subset \Gamma^k$ with $|J|\ge |\Gamma^k|-2^{j-1}$ we have
 $$
 \|f_{B_j}\| \le \|f_J\|.
 $$
 We note that this implies that if for some $Q\subset \Gamma^k$ we have
 \begin{equation}\label{3.5}
 \|f_Q\| <\|f_{B_j}\|\quad \text{then}\quad |Q|< |\Gamma^k|-2^{j-1}.
 \end{equation}

 For a given $b> 1$, to be specified later, denote by $L$ the index such that $(B_0:=\Gamma^k)$
 $$
 \|f_{B_0}\| < b\|f_{B_1}\|,
 $$
 $$
 \|f_{B_1}\| < b\|f_{B_2}\|,
 $$
 $$
 \dots
 $$
 $$
 \|f_{B_{L-2}}\| < b\|f_{B_{L-1}}\|,
 $$
 $$
 \|f_{B_{L-1}}\| \ge b\|f_{B_{L}}\|.
 $$
 Then
 \begin{equation}\label{3.6}
 \|f_{B_j}\| \le b^{L-1-j}\|f_{B_{L-1}}\|, \quad j=1,2,\dots,L.
 \end{equation}

  We now proceed to a general step. Let $m>k$ and let $A,B \subset \Gamma ^k$ be such that
 $ A=\Gamma^k\setminus B$. As above we bound $S_m$ from below. It is clear that $S_m\ge 0$.
 Denote $A_m := A\cap \Gamma^{m-1}$. Then
 $$
 S_m \ge F_{f_{m-1}}(f_{A_m}/\|f_{A_m}\|_1).
 $$
   Next,
 $$
 F_{f_{m-1}}(f_{A_m}) = F_{f_{m-1}}(f_{A_m}+f_B-f_B).
 $$
 Then $f_{A_m}+f_B=f^\e - f_\Lambda$ with $F_{f_{m-1}}(f_\Lambda)=0$. Moreover, it is easy to see that $F_{f_{m-1}}(f^\e)\ge\|f_{m-1}\|-\e$. Therefore,
 $$
 F_{f_{m-1}}(f_{A_m}+f_B-f_B)\ge \|f_{m-1}\|-\e - \|f_B\|.
 $$
 Thus
 $$
 S_m\ge \|f_{A_m}\|^{-1}_1\max(0,\|f_{m-1}\|-\e-\|f_B\|).
 $$
 By (\ref{C1}) we get
 $$
 \|f_{A_m}\|_1 \le C_1 |A_m|^{r}\|f_{A_m}\| \le C_1 |A|^{r}\|f_{A_m}\| .
 $$
 Then
 \begin{equation}\label{3.7}
 S_m \ge \frac{\|f_{m-1}\|-\|f_B\|-\e}{C_1 |A|^{r} \|f_{A_m}\|}.
 \end{equation}
 From the definition of the modulus of smoothness we have for any $\la$
$$
\|f_{m-1}-\la \varphi_m \| +\|f_{m-1}+\la \varphi_m \| \le 2\|f_{m-1}\|(1+\rho(\frac{\la}{\|f_{m-1}\|}))
$$
and by (1) from the definition of the WCGA and Lemma 6.10 from \cite{Tbook}, p. 343, we get
$$
|F_{f_{m-1}}(\varphi_m)| \ge t \sup_{g\in \D} |F_{f_{m-1}}(g)| =
$$
$$
t\sup_{\phi\in A_1(\D)} |F_{f_{m-1}}(\phi)|.
$$
 From here  we obtain
 $$
 \|f_m\| \le \|f_{m-1}\| +\inf_{\la\ge 0} (-\la t S_m + 2\|f_{m-1}\|\rho(\la/\|f_{m-1}\|)).
 $$
 We discuss here the case $\rho(u) \le \gamma u^2$. Using (\ref{3.7}) we get
 $$
 \|f_m\| \le \|f_{m-1}\|\left(1-\frac{\la t}{C_1|A|^{r}\|f_{A_m}\|} +2\gamma \frac{\la^2}{\|f_{m-1}\|^2}\right) + \frac{\la t(\|f_B\|+\e)}{C_1 |A|^{r} \|f_{A_m}\|}.
 $$
 Let $\la_1$ be a solution of
 $$
 \frac{\la t}{2C_1|A|^{r}\|f_{A_m}\|} = 2\gamma \frac{\la^2}{\|f_{m-1}\|^2},\quad \la_1 = \frac{t\|f_{m-1}\|^2}{4\gamma C_1 |A|^{r}\|f_{A_m}\|}.
 $$
 Our assumption (\ref{C2}) gives
 $$
 \|f_{A_m}\| \le U(\|f_{m-1}\|+\e).
 $$
 Specify
 $$
 \la = \frac{t \|f_{A_m}\|}{16\gamma C_1 |A|^{r} U^2}.
 $$
 Then $\la \le \la_1$ and we obtain
 \begin{equation}\label{3.8}
 \|f_m\| \le \|f_{m-1}\| \left( 1- \frac{t^2}{32\gamma C_1^2 U^2 |A|^{2r}}\right)+ \frac{ t^2 (\|f_B\|+\e)}{16\gamma C_1^2 |A|^{2r} U^2}.
 \end{equation}
 Denote $c_1:= \frac{t^2}{32\gamma C_1^2 U^2}$ and $c_2:=\frac{t^2}{16\gamma C_1^2 U^2}$.
 This implies for $m_2>m_1\ge k$
 \begin{equation*}
 \|f_{m_2}\| \le \|f_{m_1}\| (1-c_1/|A|^{2r})^{m_2-m_1} + \frac{c_2(m_2-m_1)}{|A|^{2r}}(\|f_B\|+\e).
 \end{equation*}
 Define $m_0:= k$ and, inductively,
 $$
 m_j=m_{j-1} + \beta |A_j|^{2r},\quad j=1,\dots,n.
 $$
 At iterations from $m_{j-1}+1$ to $m_j$ we use $A=A_j$ and obtain from (\ref{3.8}) that
 $$
 \|f_m\| \le \|f_{m-1}\|(1-u) + 2u(\|f_B\|+\e),\quad u:= c_1|A|^{-2r}.
 $$
 Using $1-u\le e^{-u}$ and $\sum_{k=0}^\infty (1-u)^k = 1/u$ we derive from here
 $$
 \|f_{m_j}\| \le \|f_{m_{j-1}}\|e^{-c_1\beta}+ 2 (\|f_{B_j}\|+\e).
 $$
 We continue it up to $j=L$. Denote $\eta:=e^{-c_1\beta}$. Then
 $$
 \|f_{m_L}\| \le \|f_k\|\eta^L +2 \sum_{j=1}^L (\|f_{B_j}\|+\e)\eta^{L-j}.
 $$
 We bound the $\|f_k\|$. It follows from the definition of $f_k$ that $\|f_k\|$ is the error of best approximation of $f_0$ by the subspace $\Phi_k$. Representing $f_0=f+f_0-f$ we see that $\|f_k\|$ is not greater than  the error of best approximation of $f$ by the subspace $\Phi_k$ plus $\|f_0-f\|$. This implies $\|f_k\|\le \|f_{B_0}\| +\e$. Therefore using (\ref{3.6}) we continue
 $$
 \le (\|f_{B_0}\|+\e)\eta^L +2\sum_{j=1}^L (\|f_{B_{L-1}}\|(\eta b)^{L-j} b^{-1} + \e\eta^{L-j})
 $$
 $$
 \le b^{-1}\|f_{B_{L-1}}\|\left( (\eta b)^L + 2 \sum_{j=1}^L (\eta b)^{L-j}\right) + \frac{2\e}{1-\eta}.
 $$
We will specify $\beta$ later. However, we note that it will be chosen in such a way that guarantees $\eta< 1/2$. Choose $b=\frac{1}{2\eta}$. Then
 \begin{equation}\label{3.10}
 \|f_{m_L}\| \le \|f_{B_{L-1}}\|8  e^{-c_1\beta}+4\e.
 \end{equation}

 By (\ref{C2}) we get
 $$
 \|f_{\Gamma^{m_L}}\| \le U(\|f_{m_L}\|+\e) \le U(\|f_{B_{L-1}}\|8 e^{-c_1\beta} + 5\e).
 $$
 If $\|f_{B_{L-1}}\|\le 10U \e$ then by (\ref{3.10})
 $$
 \|f_{m_L}\| \le CU\e.
 $$
 If $\|f_{B_{L-1}}\|\ge 10U\e$ then
 making $\beta$ sufficiently large to satisfy $16U e^{-c_1\beta }<1$ so that $\beta = \frac{C_3\ln (U+1)}{c_1}$,  we get
 $$
 U(\|f_{B_{L-1}}\|8e^{-c_1\beta} +5\e)< \|f_{B_{L-1}}\|
 $$
 and therefore
 $$
 \|f_{\Gamma^{m_L}}\|<\|f_{B_{L-1}}\|.
 $$
 This implies (see (\ref{3.5})
 $$
 |\Gamma^{m_L}| < |\Gamma^k| - 2^{L-2}.
 $$
   We begin with $f_0$ and apply the above argument (with $k=0$). As a result we either get the required inequality or we reduce the cardinality of support of $f$ from $|T|=K$ to $|\Gamma^{m_{L_1}}|<|T|-2^{L_1-2}$, $m_{L_1}\le \beta 2^{2rL_1}$. We continue the process and build a sequence $m_{L_j}$ such that $m_{L_j}\le \beta 2^{2rL_j}$ and after $m_{L_j}$ iterations we reduce the support by at least $2^{L_j-2}$. We also note that $m_{L_j}\le \beta 2^{2r} K^{2r}$. We continue this process till the following inequality is satisfied for the first time
 \begin{equation}\label{3.11}
  m_{L_1}+\dots+m_{L_n}  \ge 2^{4r}\beta K^{2r}.
 \end{equation}
 Then, clearly,
 $$
  m_{L_1}+\dots+m_{L_n} \le 2^{4r+1}\beta  K^{2r}.
 $$
 Using the inequality
 $$
 (a_1+\cdots +a_n)^{\theta} \le a_1^\theta+\cdots +a_n^\theta,\quad a_j\ge 0,\quad \theta\in (0,1]
 $$
 we derive from (\ref{3.11})
 $$
 2^{L_1-2}+\dots+2^{L_n-2} \ge \left(2^{2r(L_1-2)}+\dots+2^{2r(L_n-2)}\right)^{\frac{1}{2r}}
 $$
 $$
 \ge
 2^{-2}\left(2^{2rL_1}+\dots+2^{2rL_n}\right)^{\frac{1}{2r}}
 $$
 $$
 \ge 2^{-2}\left((\beta)^{-1}(m_{L_1}+\dots+m_{L_n})\right)^{\frac{1}{2r}}\ge K.
 $$
 Thus, after not more than $N:=2^{4r+1}\beta  K^{2r}$ iterations we recover $f$ exactly and then $\|f_N\| \le \|f_0-f\|\le \e$.

 \end{proof}

  Theorem \ref{T1.5} from the Introduction follows from Theorems \ref{T3.2} and \ref{T3.1}.

\section{Discussion}

We begin with presenting some known results about exact recovery and the Lebesgue-type inequalities for incoherent dictionaries.
In this case we use another natural generalization of the WOMP. This generalization of the WOMP was introduced in \cite {T23}. In the paper \cite{ST} we proved Lebesgue-type inequalities for that algorithm. We now formulate the corresponding results.  We recall a generalization of the concept of $M$-coherent dictionary to the case of Banach spaces (see, for instance, \cite{Tbook}).

Let $\D$ be a dictionary in a Banach space $X$.  The coherence parameter of this dictionary is defined as
$$
M(\D):= \sup_{g\neq h;g,h\in\D}\sup_{F_g}|F_g(h)|.
$$
In general, a norming functional $F_g$ is not unique. This is why we take $\sup_{F_g}$ over all norming functionals of $g$ in the definition of $M(\D)$. We do not need $\sup_{F_g}$ in the definition of $M(\D)$ if for each
$g\in\D$ there is a unique norming functional $F_g\in X^*$. Then we define $\D^*:=\{F_g,g\in\D\}$ and call $\D^*$ a {\it dual dictionary} to a dictionary $\D$.
 It is known that the uniqueness of the norming functional $F_g$ is equivalent to the property that $g$ is a point of Gateaux smoothness:
 $$
 \lim_{u\to 0}(\|g+uy\|+\|g-uy\|-2\|g\|)/u =0
 $$
 for any $y\in X$. In particular, if $X$ is uniformly smooth then $F_f$ is unique for any $f\neq 0$. We considered in \cite{T23} the following greedy algorithm which generalizes the Weak Orthogonal Greedy Algorithm to a Banach space setting.

 {\bf Weak Quasi-Orthogonal Greedy Algorithm (WQOGA).}
 Let $t\in (0,1]$ and $f_0$ be given.   Find $\varphi_1:=\varphi_1^{q,t}\in\D$ (here and below index $q$ stands for {\it quasi-orthogonal}) such that
 $$
 |F_{\varphi_1}(f_0)| \ge t\sup_{g\in \D}|F_g(f_0)|.
 $$
 Next, we find $c_1$ satisfying
 $$
 F_{\varphi_1}(f-c_1\varphi_1)=0.
 $$
 Denote $f_1:=f_1^{q,t}:=f-c_1\varphi_1$.

 We continue this construction in an inductive way. Assume that we have already constructed residuals $f_0,f_1,\dots,f_{m-1}$ and dictionary elements $\varphi_1,\dots,\varphi_{m-1}$. Now, we pick an element
 $\varphi_m:=\varphi_m^{q,t}\in\D$ such that
 $$
 |F_{\varphi_m}(f_{m-1})| \ge t\sup_{g\in \D}|F_g(f_{m-1})|.
 $$
 Next, we look for $c_1^m,\dots,c_m^m$ satisfying
 \begin{equation}\label{1.1}
 F_{\varphi_j}(f-\sum_{i=1}^mc_i^m\varphi_i)=0,\quad j=1,\dots,m.
 \end{equation}
If there is no solution to (\ref{1.1}) then we stop, otherwise we denote $G_m:=G_m^{q,t} := \sum_{i=1}^mc_i^m\varphi_i$ and $f_m:=f_m^{q,t}:=f-G_m $ with $c_1^m,\dots,c_m^m$ satisfying (\ref{1.1}).

\begin{Remark}\label{R1.1} Note that (\ref{1.1}) has a unique solution if
  $\det (F_{\varphi_j}(\varphi_i))_{i,j=1}^m \neq 0$. Applying the WQOGA in the case of a dictionary with the coherence parameter $M:=M(\D)$ gives, by a simple well known argument on the linear independence of the rows of the matrix $(F_{\varphi_j}(\varphi_i))_{i,j=1}^m$, the conclusion  that (\ref{1.1}) has a unique solution for any $m<1+1/M$.
 Thus, in the case of an $M$-coherent
dictionary $\D$, we can run the WQOGA for at least $[1/M]$ iterations.
\end{Remark}

In the case $t=1$ we call the WQOGA the Quasi-Orthogonal Greedy Algorithm (QOGA). In the case of QOGA we need to make an extra assumption that the corresponding maximizer $\ff_m\in\D$ exists. Clearly, it is the case when $\D$ is finite.

It was proved in \cite{T23} (see also \cite{Tbook}, p. 382) that the WQOGA is as good as the WOMP in the sense of exact recovery of sparse signals with respect to incoherent dictionaries. The following result was obtained in \cite{T23}.
 \begin{Theorem}\label{T1.2} Let $t\in (0,1]$. Assume that $\D$ has coherence parameter $M$. Let $K<\frac{t}{1+t}(1+1/M)$. Then for any
 $f_0$ of the form
 \begin{equation*}
 f_0=\sum_{i=1}^Ka_ig_i,
 \end{equation*}
 where $g_i$ are distinct elements of $\D$, the WQOGA recovers it exactly after $K$ iterations. In other words, $f^{q,t}_K=0$.
 \end{Theorem}
It is known (see  \cite{Tbook}, pp. 303--305) that the bound $K<\frac{1}{2}(1+1/M)$ is sharp for exact recovery by the OGA.

We introduce a new norm, associated with a dictionary $\D$,  by the formula
$$
\|f\|_\D:=\sup_{g\in\D}|F_g(f)|,\quad f\in X.
$$
  We define best $m$-term approximation in the norm $Y$ as follows
$$
\sigma_m(f)_Y := \inf_{g\in\Sigma_m(\D)}\|f-g\|_Y.
$$
In \cite{ST} the norm $Y$ was either the norm $X$ of our Banach space or the norm $\|\cdot\|_\D$ defined above.
 The following two Lebesgue-type inequalities were proved in \cite{ST}.
 \begin{Theorem}\label{T1.3} Assume that $\D$ is an $M$-coherent dictionary. Then for \newline $m\le \frac{1}{3M}$ we have for the QOGA
 \begin{equation}\label{1.3}
 \|f_m\|_\D \le 13.5\sigma_m(f)_\D.
 \end{equation}
 \end{Theorem}

 \begin{Theorem}\label{T1.4} Assume that $\D$ is an $M$-coherent dictionary in a Banach space $X$. There exists an absolute constant $C$   such that, for $m\le 1/(3M)$, we have for the QOGA
$$
\|f_m\|_X \le C \inf_{g\in\Sigma_m(\D)}(\|f-g\|_X + m\|f-g\|_\D).
$$
\end{Theorem}
\begin{Corollary}\label{C1.1} Using the inequality $\|g\|_\D \le \|g\|_X$,  Theorem \ref{T1.4} obtains
$$
\|f_m\|_X \le C(1+m) \sigma_m(f)_X.
$$
\end{Corollary}

Inequality (\ref{1.3}) is a perfect (up to a constant 13.5) Lebesgue-type inequality. It indicates that the norm $\|\cdot\|_\D$ used in \cite{ST} is a suitable norm for analyzing performance of the QOGA. Corollary \ref{C1.1} shows that the Lebesgue-type inequality (\ref{1.3}) in the norm $\|\cdot\|_\D$ implies the Lebesgue-type inequality in the norm $\|\cdot\|_X$.

Thus, results of this paper complement the above discussed results from \cite{T23} and \cite{ST}. Results from \cite{T23} and \cite{ST} deal with incoherent dictionaries and use the QOGA for exact recovery and the Lebesgue-type inequalities. Results of this paper deal with dictionaries which satisfy assumptions {\bf A1} and {\bf A2} and we analyze the WCGA here. In the case of a Hilbert space, assumptions {\bf A1} and {\bf A2} are satisfied if $\D$ has RIP. It is well known that the RIP condition is much weaker than the incoherence condition in the case of a Hilbert space. It is interesting to note that we do not know how the coherence parameter $M(\D)$ is related to properties {\bf A1} and {\bf A2} in the case of a Banach space.

We now give a few applications of Theorem \ref{T1.5} for specific dictionaries $\D$. We begin with the case when $\D$ is a basis $\Psi$ for $X$. In some of our examples we take $X=L_p$, $2\le p<\infty$. Then it is known that $\rho(u) \le \gamma u^2$ with $\gamma = (p-1)/2$.

{\bf Example 1.} Let $X$ be a Banach space with $\rho(u)\le \gamma u^2$ and with cotype $q$.
Let $\Psi$ be a normalized in $X$ unconditional basis for $X$. Then $U\le C(X,\Psi)$. By Remark \ref{R3.1} $\Psi$ satisfies {\bf A1} with $r=1-\frac{1}{q}$. Theorem \ref{T1.5} gives
\begin{equation}\label{4.3}
\|f_{C(t,X,\Psi)K^{2-2/q}}\| \le C\sigma_K(f_0,\Psi) .
\end{equation}
We note that (\ref{4.3}) provides some progress in Open Problem 7.1 (p. 91) from \cite{Tsurv}.

{\bf Example 2.} Let $\Psi$ be a uniformly bounded orthogonal system normalized in $L_p(\Omega)$, $2\le p<\infty$, $\Omega$ is a bounded domain. Then we can take $r=1/2$. The inequality
$$
\|g\|_p \le CK^{1/2-1/p}\|g\|_2
$$
for $K$-sparse $g$ implies that
$$
\|S_A(f)\|_p \le CK^{1/2-1/p}\|S_A(f)\|_2 \le CK^{1/2-1/p}\|f\|_2 \le CK^{1/2-1/p}\|f\|_p.
$$
Therefore $U\le CK^{1/2-1/p}$. Theorem \ref{T1.5} gives
 \begin{equation}\label{4.4}
\|f_{C(t,p,D)K^{2/p'}\ln K}\|_p \le C\sigma_K(f_0,\Psi)_p .
\end{equation}
Inequality (\ref{4.4}) provides some progress in Open Problem 7.2 (p. 91) from \cite{Tsurv}.

Theorem \ref{T1.5} can also be applied for quasi-greedy bases and other greedy-type bases (see \cite{Tbook}).
We plan to discuss these applications in detail in our future work.

In this paper we limit ourselves to the case of Banach spaces satisfying the condition $\rho(u)\le \gamma u^2$.
In particular, as we mentioned above the $L_p$ spaces with $2\le p<\infty$ satisfy this condition. Clearly, the $L_p$ spaces with $1<p\le 2$ are also of interest. For the clarity of presentation we do not discuss the case $\rho(u)\le \gamma u^q$ in this paper.
The technique from Section 3 works in this case too and we will present the corresponding results in our future work.

\end{document}